\documentclass[11pt]{amsart}

\UseRawInputEncoding  
\usepackage{amsmath,amssymb,amsthm}
\usepackage{amsmath,amssymb,amsthm,amscd}
\usepackage[frame,cmtip,arrow,matrix,line,graph,curve]{xy}
\usepackage{graphpap, color}
\usepackage[mathscr]{eucal}
\usepackage{color}

\def\dfrac{\displaystyle\frac}
\def\dsum{\displaystyle\sum}

\newtheorem{prop}{Proposition}
\newtheorem{theo}[prop]{Theorem}
\newtheorem{lemm}[prop]{Lemma}

\newtheorem{rema}[prop]{Remark}

\newtheorem{defi}[prop]{Definition}

\renewcommand{\leq}{\leqslant}
\renewcommand{\geq}{\geqslant}

\numberwithin{equation}{section}

\title{Interior Hessian estimates for sum Hessian quotient equation}

\begin{document}

\author{Changyu Ren}
\address{School of Mathematical Science\\
Jilin University\\ Changchun\\ China}
\email{rency@jlu.edu.cn}
\author{Ziyi Wang}
\address{School of Mathematical Science\\ Jilin University\\ Changchun\\ China}
\email{ziyiw23@mails.jlu.edu.cn}
\thanks{Research of the first author is supported by NSFC Grant No. 12571216, 11871243.}
\begin{abstract}
This paper is devoted to the interior $C^2$ estimates for a
class of sum Hessian quotient equations. For $0\leq l<k<n$, we establish the interior estimates
and the Pogorelov type estimates. In the case $k=n$, we obtain a
weaker Pogorelov type estimate for $0\leq l<n-1$.
\end{abstract}
\maketitle

\section{introduction}
\par
In this paper, we mainly study the interior $C^2$ estimates for the
following sum Hessian quotient equation
\begin{equation}\label{e1.1}
\frac{\sigma_k(\eta)+\alpha\sigma_{k-1}(\eta)}{\sigma_l(\eta)+\alpha\sigma_{l-1}(\eta)}=f(x,u,Du),
\end{equation}
where $u$ is an unknown function defined on $\Omega$, $f$ is a
positive function, $\alpha\geq0$ and $0\leq l<k\leq n$. 
$\sigma_k(D^2u)=\sigma_k(\lambda(D^2u))$ denote the $k$-th
elementary symmetric function
\begin{equation*}
\sigma_k(\lambda)=\sum_{1\leq i_1<\cdots<i_k\leq
n}\lambda_{i_1}\cdots\lambda_{i_k}.
\end{equation*}
of the eigenvalues of the Hessian
matrix $D^2u$ for $\lambda=(\lambda_1,\cdots,\lambda_n)\in\mathbb R^n$. Define that~$\eta=(\eta_1,\eta_2,\cdots,\eta_n)$~
with~$\eta_i=\dsum_{j\neq i}\lambda_j$~, then $\eta$ are the
eigenvalues of the matrix $(\Delta u)I-D^2u$.

\par
The operator $\sigma_k$ with the particular matrix $(\Delta u)I-D^2u$ inside has been studied 
in many literatures. It originates from the Gauduchon conjecture \cite{G3,
STW} in complex geometry. In \cite{HL1, HL2}, Harvey-Lawson introduced
and studied the $(n-1)$-plurisubharmonic functions, solving the 
Dirichlet problem with $f=0$ on suitable domains, where the $(n-1)$-plurisubhar-
monic functions satisfy the nonnegative definite complex matrix $(\Delta u)I-D^2u$. 
Furthermore, this operator can also be viewed as a form-type Calabi-Yau equation in the 
work of Fu-Wang-Wu \cite{FWW}, and Tosatti-Weinkove \cite{TW2}
solved that on K$\ddot{\rm a}$hler manifolds.

\par
Another application for the equation $\sigma_k(\eta)=f$ lies in its connection to $p$-convex
hypersurface in prescribed curvature problem. Suppose that $M$ is a
hypersurface in ${\mathbb R}^{n+1}$. For given $1\leq p\leq n$, call
$M$ $p$-convex if $\kappa(X)$ satisfy
$\kappa_{i_1}+\cdots+\kappa_{i_p}\geq0$ for each $X\in M$, where
$\kappa(X)=(\kappa_1,\cdots,\kappa_n)$ are the principal curvatures
of $M$. Chu-Jiao \cite{CJ} established the existence
of hypersurfaces satisfying $\sigma_k(\eta)=f$ in
${\mathbb R}^{n+1}$ for prescribed curvature problem. Dong \cite{D2} proved 
the existence of $p$-convex hypersurface with prescribed curvature under the condition  $p\geq\frac{n}{2}$. For the general case
$1\leq p\leq n$, the $p$-convex hypersurface has been extensively investigated 
by Sha \cite{S2, S3}, Wu \cite{W} and Harvey-Lawson
\cite{HL5}.

\par
The variant forms of the equation $\sigma_k(\eta)=f$ have also been studied. One example is the Hessian quotient
equation $\frac{\sigma_k(\eta)}{\sigma_l(\eta)}=f$. For instance, in the case where $l>0$ and $l+2\leq k\leq n$, 
Chen-Tu-Xiang \cite{CTX} established the Pogorelov type estimates for such equation. Moreover, Dong-Wei \cite{DW} showed that the
Hessian quotient equation is strictly elliptic within the cone
$\Gamma_{k+1}^{\prime}$ (see Definition 1 below). Furthermore, for the case $0\leq l<k<n$, Chen-Dong-Han
\cite{CDH} derived both interior estimates and Pogorelov type
estimates.

\par
Replacing $\eta$ with $\lambda(D^2u)$ turns the operator
$\sigma_k(\eta)$ into the following classic $k$-Hessian
equation
\begin{equation*}
\sigma_k(D^2u)=f(x,u,Du).
\end{equation*}
The existence of solutions to this equation is generally
established via the continuity method and prior estimates. 
Caffarelli-Nirenberg-Spruck \cite{CNS3} resolved the Dirichlet
problem when the right hand side function $f$ is independent of the gradient term. 
However, for $f$ depending on the
gradient term, deriving $C^2$ estimates of the solution
remains a long-standing challenge. There have been some works contributed to
this issue, including \cite{CW, GLL, JW, Q, WY}.

In recent years, there has been increasingly attention on the fully nonlinear equations arising from linear
combinations of elementary symmetric functions. For example, the following
curvature equations and complex Hessian equations
$$\dsum_{s=0}^k\alpha_s\sigma_k(\kappa(\chi))=f(\chi,\upsilon(\chi)),\quad
\chi\in M,$$
$$\dsum_{s=0}^k\alpha_s\sigma_k(\lambda)=f(z,u,Du)$$
were investigated by Li-Ren-Wang \cite{LRW} and Dong \cite{Dws}. For the
 curvature equations
$$\sigma_k(W_u(x))+\alpha
\sigma_{k-1}(W_u(x))=\dsum_{l=0}^{k-2}\alpha_l(x)\sigma_l(W_u(x)),\quad
x\in\mathbb{S}^n,$$ Guan-Zhang \cite{GZ} and Zhou \cite{Zhou}
studied curvature estimates and interior gradient estimates
respectively. Recently, Liu and Ren \cite{LR} established the
Pogorelov type $C^{2}$ estimations for $(k-1)$-convex and $k$-convex
solutions of the following sum Hessian equations
\begin{equation*}
 \sigma_k(\lambda(u_{ij})) + \alpha\sigma_{k-1}(\lambda(u_{ij})) =
f(x, u,\nabla u).
\end{equation*}

\par
Motivated by the work of Chen-Dong-Han \cite{CDH} and Chen-Tu-Xiang
\cite{CTX}, the interior estimates and the Pogorelov type
estimates for a class of sum Hessian equations of the form (\ref{e1.1}) were obtained for the specific case $l=0$ in \cite{RW}. 
We now proceed to study these estimates in the broader context of a general $l$.

\par
First, we give the definition of $k$-convex and $k$-admissible.

\par
\begin{defi}\label{defi1}
For a domain $\Omega\subset{\mathbb R}^n$, a function $u\in
C^2(\Omega)$ is called $k$-convex if the eigenvalues
$\lambda(x)=(\lambda_1(x),\cdots,\lambda_n(x))$ of the Hessian
$\nabla^2u(x)$ are in $\Gamma_k$ for all $x\in\Omega$, where
$\Gamma_k$ is the Garding's cone
\begin{equation*}
\Gamma_k=\{\lambda\in{\mathbb R}^n|\sigma_m(\lambda)>0,
m=1,\cdots,k\},
\end{equation*}
Similarly, a function $u\in C^2(\Omega)$ is called $k$-admissible if
the eigenvalues $\lambda(x)=(\lambda_1(x),\cdots,\lambda_n(x))$ of
the Hessian $\nabla^2u(x)$ are in $\Gamma_k^{\prime}$ for all
$x\in\Omega$, where
\begin{equation*}
\Gamma_k^{\prime}=\{\lambda\in{\mathbb R}^n|\sigma_m(\eta)>0,
m=1,\cdots,k\},
\end{equation*}
\end{defi}

\par
Here we list the main results.

\par
\begin{theo}\label{theo1}
Suppose that $u\in C^4(B_R(0))$ is a $(k-1)$-admissible solution of
the sum Hessian equation (\ref{e1.1}), where $B_R(0)$ is a ball
centered at the origin with radius $R$ in ${\mathbb R}^n$, $0\leq l<k<n$
and $f\in C^2(B_R(0)\times\mathbb{R}\times\mathbb{R}^n)$ with
$0<m\leq f\leq M$. Then
\begin{equation}\label{e1.2}
|D^2u(0)|\leq C\big(1+\frac{\sup|Du|}{R}\big),
\end{equation}
where $C$ is a positive constant depending only on $n$, $k$, $l$, $m$, $M$,
$R\sup|Df|$ and $R^2\sup|D^2f|$.
\end{theo}

\par
Following the argument of the proof of Theorem \ref{theo1}, we can
also obtain the Pogorelov type estimates for the Dirichlet problem
\begin{equation}\label{e1.04}
\begin{cases}
\dfrac{\sigma_k(\eta)+\alpha\sigma_{k-1}(\eta)}{\sigma_l(\eta)+\alpha\sigma_{l-1}(\eta)}=f(x,u,Du),&in~~\Omega,\\
u=0,&on~~\partial \Omega.
\end{cases}
\end{equation}
\par
\begin{theo}\label{theo2}
Suppose that $u\in C^4(\Omega)\cap C^2(\overline\Omega)$ is a
$(k-1)$-admissible solution of the problem (\ref{e1.04}) in a
bounded domain $\Omega\subset{\mathbb R}^n$, where $0\leq l<k<n$ and $f\in
C^2(\overline\Omega\times\mathbb{R}\times\mathbb{R}^n)$ with $f>0$.
Then
\begin{equation*}
(-u)|D^2u|\leq C,
\end{equation*}
where $C$ depends only on $n$, $k$, $l$, $m$, $|f|_{C^2}$ and $|u|_{C^1}$.
\end{theo}

\par
For the case $k=n$, we still have the following result.

\begin{theo}\label{theo3}
Suppose that $u\in C^4(\Omega)\cap C^2(\overline\Omega)$ is a
$(n-1)$-admissible solution of the problem (\ref{e1.04}) in a
bounded domain $\Omega\subset{\mathbb R}^n$, where $k=n$, $0\leq l<n-1$ and $f\in
C^2(\overline\Omega\times\mathbb{R}\times\mathbb{R}^n)$ with $f>0$.
Then
\begin{equation*}
(-u)^{\beta}|D^2u|\leq C,
\end{equation*}
where $C$ depends only on $n$, $l$, $|f|_{C^2}$ and $|u|_{C^1}$.
\end{theo}
\par

The organization of our paper is as follows. In Section 2, we give
the preliminary knowledge. The proofs of Theorem \ref{theo1} and
Theorem \ref{theo3} are given respectively in Section 3 and Section
4.

\section{Preliminary}
We recall some basic definition and properties of elementary
symmetric function as follows.

\par
The $k^{\text{th}}$ elementary symmetric
function is defined as
\begin{eqnarray*}
\sigma_k(\lambda)=\sum_{1\leq i_1<\cdots<i_k\leq
n}\lambda_{i_1}\cdots \lambda_{i_k}
\end{eqnarray*}
for $\lambda=(\lambda_1,\lambda_2,\cdots,\lambda_n)\in\mathbb{R}^n$
and $1\leq k\leq n$. For convenience, we further set
$\sigma_0(\lambda)=1$ and $\sigma_k(\lambda)=0$ for $k>n$ or $k<0$.

\par
We define the sum Hessian function as
\begin{eqnarray*}
S_k(\lambda)=\sigma_k(\lambda)+\alpha\sigma_{k-1}(\lambda)
\end{eqnarray*}
for $\lambda=(\lambda_1,\lambda_2,\cdots,\lambda_n)\in\mathbb{R}^n$
and $1\leq k\leq n$.

Now we will list some algebraic identities and properties of
$\sigma_k$ and $S_k$. Denote
$(\lambda|a)=(\lambda_1,\cdots,\lambda_{a-1},\lambda_{a+1},\cdots,\lambda_n)$.

\begin{prop}\label{lem2.1}We have
\par
(1) $S_k^{pp}(\lambda):=\frac{\partial
S_k(\lambda)}{\partial\lambda_p}=\sigma_{k-1}(\lambda|p)+\alpha\sigma_{k-2}(\lambda|p)=S_{k-1}(\lambda|p),\quad
p=1,2,\cdots,n$;
\par
(2)
$S_k^{pp,qq}(\lambda):=\frac{\partial^2S_k(\lambda)}{\partial\lambda_p\partial\lambda_q}=S_{k-2}(\lambda|pq),\quad
p,q=1,2,\cdots,n$;
\par
(3) $S_k(\lambda)=\lambda_i S_{k-1}(\lambda|i)+S_k(\lambda|i),\quad
i=1,2,\cdots,n$;
\par
(4) $\dsum_{i=1}^n
S_k(\lambda|i)=(n-k)S_k(\lambda)+\alpha\sigma_{k-1}(\lambda)$;
\par
(5) $\dsum_{i=1}^n\lambda_i
S_{k-1}(\lambda|i)=kS_k(\lambda)-\alpha\sigma_{k-1}(\lambda)$.
\end{prop}

\begin{proof}
(1)-(5) are obvious by direct calculating.
\end{proof}

The following lemma is due to Newton \cite{N} and Maclaurin \cite{M}. This
inequality is widely used in fully nonlinear partial differential equations and
geometric analysis. 
\begin{lemm}\label{lem2.2}
If~$\lambda=(\lambda_1,\cdots,\lambda_n)\in \Gamma_n$, then
\begin{equation*}
\big[\frac{\sigma_k(\lambda)}{C_n^k}\big]^2\geq\frac{\sigma_{k-1}(\lambda)\sigma_{k+1}(\lambda)}{C_n^{k-1}C_n^{k+1}}, \quad k=1,2,\cdots,n-1,
\end{equation*}
where $C_n^k=\frac{n!}{k!(n-k)!}$, and the inequality is strict unless all entries of $\lambda$ coincide.
\end{lemm}

\par
For sum Hessian operator~$S_k(\lambda)$~, Li-Ren-Wang \cite{LRW}
have proved its admissible solution set:
\begin{equation*}
\tilde\Gamma_k=\Gamma_{k-1}\cap\{\lambda|S_k>0\}
\end{equation*}
and ~$S_k^{\frac{1}{k}}(\lambda)$~is concave in the set
~$\tilde\Gamma_k$.

\par

\begin{lemm}\label{lem2.3}We have:
\par
(1)~$\tilde\Gamma_k$~are convex cones,
and~$\tilde\Gamma_1\supset\tilde\Gamma_2\supset\cdots\supset\tilde\Gamma_n$;
\par
(2) If~$\lambda=(\lambda_1,\cdots,\lambda_n)\in
\tilde\Gamma_k$~and~$\lambda_1\geq \lambda_2\geq\cdots\geq
\lambda_n$, then~$\lambda_{k-1}>0$,
\begin{equation*}
S_{k-1}(\lambda|n)\geq S_{k-1}(\lambda|n-1)\geq\cdots\geq
S_{k-1}(\lambda|1)>0,
\end{equation*}
and
\begin{equation*}
S_{k-1}(\lambda|k)\geq c(n,k)S_{k-1}(\lambda),
\end{equation*}
where~$c(n,k)$~is a positive constant only depending on~$n$~and~$k$;
\par
(3)If~$\lambda=(\lambda_1,\cdots,\lambda_n)\in \tilde\Gamma_k$, then
for any~$(\xi_1,\cdots,\xi_n)$, we have
\begin{equation*}
\sum_{p,q}\frac{\partial^2[\frac{S_k(\lambda)}{S_l(\lambda)}]}{\partial\lambda_p\partial\lambda_q}\xi_p\xi_q\leq(1-\frac{1}{k-l})\frac{\Big[\Sigma_p
\frac{\partial[\frac{S_k(\lambda)}{S_l(\lambda)}]}{\partial\lambda_p}\xi_p\Big]^2}{\frac{S_k(\lambda)}{S_l(\lambda)}}.
\end{equation*}
\end{lemm}

\begin{proof}
The proof of (1) is obvious. The proof of (2) can be found in
\cite{R}.
\par
Now we will prove (3). Differentiating ~$[\frac{S_k(\lambda)}{S_l(\lambda)}]^\frac{1}{k-l}$~ yields
\begin{equation*}
\frac{\partial[\frac{S_k(\lambda)}{S_l(\lambda)}]^\frac{1}{k-l}}{\partial\lambda_p}=\frac{1}{k-l}[\frac{S_k(\lambda)}{S_l(\lambda)}]^{\frac{1}{k-l}-1}\frac{\partial[\frac{S_k(\lambda)}{S_l(\lambda)}]}{\partial\lambda_p}.
\end{equation*}
Differentiating twice, we get
\begin{align*}
\frac{\partial^2[\frac{S_k(\lambda)}{S_l(\lambda)}]^\frac{1}{k-l}}{\partial\lambda_p\partial\lambda_q}
=&\frac{1}{k-l}(\frac{1}{k-l}-1)[\frac{S_k(\lambda)}{S_l(\lambda)}]^{\frac{1}{k-l}-2}\frac{\partial[\frac{S_k(\lambda)}{S_l(\lambda)}]}{\partial\lambda_p}\frac{\partial[\frac{S_k(\lambda)}{S_l(\lambda)}]}{\partial\lambda_q}\\
&+\frac{1}{k-l}[\frac{S_k(\lambda)}{S_l(\lambda)}]^{\frac{1}{k-l}-1}\frac{\partial^2[\frac{S_k(\lambda)}{S_l(\lambda)}]}{\partial\lambda_p\partial\lambda_q}.
\end{align*}
Notice that ~$[\frac{S_k(\lambda)}{S_l(\lambda)}]^\frac{1}{k-l}$~ is concave in
~$\tilde\Gamma_k$, so
\begin{equation*}
\sum_{p,q}\frac{\partial^2[\frac{S_k(\lambda)}{S_l(\lambda)}]^\frac{1}{k-l}}{\partial\lambda_p\partial\lambda_q}\xi_p\xi_q\leq0.
\end{equation*}
By simplification we finally have
\begin{equation*}
\sum_{p,q}\frac{\partial^2[\frac{S_k(\lambda)}{S_l(\lambda)}]}{\partial\lambda_p\partial\lambda_q}\xi_p\xi_q\leq(1-\frac{1}{k-l})\frac{\Big[\Sigma_p
\frac{\partial[\frac{S_k(\lambda)}{S_l(\lambda)}]}{\partial\lambda_p}\xi_p\Big]^2}{\frac{S_k(\lambda)}{S_l(\lambda)}}.
\end{equation*}
\end{proof}

For~$\lambda=(\lambda_1,\cdots,\lambda_n)$~, note
that~$\eta=(\eta_1,\eta_2,\cdots,\eta_n)$~
with~$\eta_i=\sum\limits_{j\neq i}\lambda_j$~.

\par
Similarly, we define the cone
\begin{equation*}
\Gamma_k^{\prime}=\{\lambda=(\lambda_1,\cdots,\lambda_n):\eta\in\tilde\Gamma_k\}.
\end{equation*}
We now list some basic properties
of~$S_k(\eta)$~with~$\lambda\in\Gamma_k^{\prime}$.

\begin{lemm}\label{lem2.4}
Suppose that~$\lambda=(\lambda_1,\cdots,\lambda_n)\in
\Gamma_k^{\prime}$~,then~$S_k^\frac{1}{k}(\eta)$~is concave
for~$\eta$~. Thus, for~$(\xi_1,\cdots,\xi_n)$~we have
\begin{equation*}
\sum_{p,q}\frac{\partial^2[\frac{S_k(\eta)}{S_l(\eta)}]}{\partial\lambda_p\partial\lambda_q}\xi_p\xi_q\leq(1-\frac{1}{k-l})\frac{\Big[\Sigma_p
\frac{\partial[\frac{S_k(\eta)}{S_l(\eta)}]}{\partial\lambda_p}\xi_p\Big]^2}{\frac{S_k(\eta)}{S_l(\eta)}}.
\end{equation*}
\end{lemm}

\begin{proof}
From Lemma~\ref{lem2.3} (3), we have
\begin{align*}
\sum_{p,q}\frac{\partial^2[\frac{S_k(\eta)}{S_l(\eta)}]}{\partial\lambda_p\partial\lambda_q}\xi_p\xi_q
=&\sum_{p,q,a,b}\frac{\partial^2[\frac{S_k(\eta)}{S_l(\eta)}]}{\partial\eta_a\partial\eta_b}\frac{\partial\eta_a}{\partial\lambda_p}\frac{\partial\eta_b}{\partial\lambda_q}\xi_p\xi_q\\
\leq&(1-\frac{1}{k-l})\frac{\Big[\Sigma_a\frac{\partial
[\frac{S_k(\eta)}{S_l(\eta)}]}{\partial\eta_a}(\Sigma_p\frac{\partial\eta_a}{\partial\lambda_p}\xi_p)\Big]^2}{\frac{S_k(\eta)}{S_l(\eta)}}\\
=&(1-\frac{1}{k-l})\frac{\Big[\Sigma_p\frac{\partial
[\frac{S_k(\eta)}{S_l(\eta)}]}{\partial\lambda_p}\xi_p\Big]^2}{\frac{S_k(\eta)}{S_l(\eta)}}.
\end{align*}
\end{proof}

\par

\begin{lemm}\label{lem2.5}
If~$\lambda=(\lambda_1,\cdots,\lambda_n)\in \tilde\Gamma_k$, $1\leq
l<k\leq n$, then
\begin{equation}\label{e2.4}
\frac{\sigma_{k-1}(\lambda)\sigma_l(\lambda)}{C_n^{k-1}C_n^l}\geq\frac{\sigma_k(\lambda)\sigma_{l-1}(\lambda)}{C_n^kC_n^{l-1}},
\end{equation}
\begin{equation}\label{e2.5}
\frac{S_{k-1}(\lambda)S_l(\lambda)}{C_{n+1}^{k-1}C_{n+1}^l}\geq\frac{S_k(\lambda)S_{l-1}(\lambda)}{C_{n+1}^kC_{n+1}^{l-1}}.
\end{equation}
\end{lemm}
\begin{proof}
First we prove (\ref{e2.4}). For the case $\sigma_k(\lambda)\leq0$, the inequality is obviously
correct. If $\sigma_k(\lambda)>0$, from $\lambda\in\Gamma_{k-1}$ we have
$\lambda\in\Gamma_k$. By using Lemma \ref{lem2.2} we complete the proof.
\par
Note that $S_k(\lambda)=\sigma_k(\lambda^*)$, where $\lambda^*=(\lambda_1,\cdots,\lambda_n,\alpha)$. Combining with (\ref{e2.4}), we derive (\ref{e2.5}).
\end{proof}

\par

\begin{lemm}\label{lem2.6}
Suppose that~$\lambda=(\lambda_1,\cdots,\lambda_n)\in
\tilde\Gamma_k$, $0\leq l<k$ and $0\leq q<p$, then
\begin{equation*}
(\frac{S_k}{S_l})^{\frac{1}{k-l}}\leq(\frac{S_p}{S_q})^{\frac{1}{p-q}}
\end{equation*}
for $k\geq p$ and $l\geq q$.
\end{lemm}
\begin{proof}
The proof follows a similar approach to that in \cite{S4}. From (\ref{e2.5}) we have $\frac{S_{i+1}}{S_i}\leq\frac{S_i}{S_{i-1}}$ for $1\leq
i<k$. We use induction to complete the proof.
\par
For the case $k=1$, the inequality holds obviously. For $k\geq2$,
suppose the inequality holds, then
\begin{align*}
(\frac{S_{k+1}}{S_l})^{\frac{1}{k+1-l}}=&(\frac{S_{k+1}}{S_k}\frac{S_k}{S_l})^{\frac{1}{k+1-l}}\leq(\frac{S_k}{S_{k-1}}\frac{S_k}{S_l})^{\frac{1}{k+1-l}}\nonumber\\
\leq&\big[(\frac{S_p}{S_q})^{\frac{1}{p-q}}(\frac{S_p}{S_q})^{\frac{k-l}{p-q}}\big]^{\frac{1}{k+1-l}}=(\frac{S_p}{S_q})^{\frac{1}{p-q}}.
\end{align*}
On the other hand, if $p=k+1$, then
\begin{equation*}
(\frac{S_{k+1}}{S_l})^{\frac{1}{k+1-l}}\leq(\frac{S_{k+1}}{S_q})^{\frac{1}{k+1-q}}
\end{equation*}
is equivalent to
\begin{equation*}
(\frac{S_{k+1}}{S_q})^{\frac{1}{k+1-q}}\leq(\frac{S_l}{S_q})^{\frac{1}{l-q}}.
\end{equation*}
Since $l<k$, we complete the proof of induction.
\end{proof}

\par

\begin{lemm}\label{lem2.7}
Suppose that~$\lambda=(\lambda_1,\cdots,\lambda_n)\in
\Gamma_k^{\prime}$, and there is an ordering that~$\lambda_1\geq
\lambda_2\geq\cdots\geq \lambda_n$, then we have
\par
(1)~$\eta_1\leq\eta_2\leq\cdots\leq\eta_n$, and~$\eta_{n-k+2}>0$;
\par
(2)~$S_{k-1}(\eta|n-k+1)\geq \theta(n,k) S_{k-1}(\eta)$,
~$0<k<n$;
\par
(3)~$\frac{\partial\big[\frac{S_k(\eta)}{S_l(\eta)}\big]}{\partial\eta_1}\geq\frac{\partial
\big[\frac{S_k(\eta)}{S_l(\eta)}\big]}{\partial\eta_2}\geq\cdots\geq\frac{\partial
\big[\frac{S_k(\eta)}{S_l(\eta)}\big]}{\partial\eta_n}$,  $0\leq
l<k<n$;
\par
(4)~$\frac{\partial\big[\frac{S_k(\eta)}{S_l(\eta)}\big]}{\partial\lambda_1}\leq\frac{\partial
\big[\frac{S_k(\eta)}{S_l(\eta)}\big]}{\partial\lambda_2}\leq\cdots\leq\frac{\partial
\big[\frac{S_k(\eta)}{S_l(\eta)}\big]}{\partial\lambda_n}$, 
$0\leq l<k<n$;
\par
(5)~$\forall 1\leq i\leq n,\frac{\partial
\big[\frac{S_k(\eta)}{S_l(\eta)}\big]}{\partial\lambda_i}\geq
c(n,k,l)\sum_j\frac{\partial
\big[\frac{S_k(\eta)}{S_l(\eta)}\big]}{\partial\lambda_j}$;
\par
(6)~$\sum_i\frac{\partial
\big[\frac{S_k(\eta)}{S_l(\eta)}\big]}{\partial\lambda_i}=(n-1)\sum_i\frac{\partial\big[\frac{S_k(\eta)}{S_l(\eta)}\big]}{\partial\eta_i}\geq
c(n,k,l)f^{1-\frac{1}{k-l}}$ for equation(\ref{e1.1}).
\end{lemm}

\begin{proof}
From Lemma~\ref{lem2.2}, we can directly get~(1)~and~(2). The proof of~(3)~and~(4)~are obvious. Now we prove~(5). From~(4), we
only need to prove inequality where~$i=1$. We have
\begin{align*}
\sum_j\frac{\partial \big[\frac{S_k(\eta)}{S_l(\eta)}\big]}{\partial\lambda_j}=&\sum_j\sum_{p\neq j}\frac{S_{k-1}(\eta|p)S_l(\eta)-S_k(\eta)S_{l-1}(\eta|p)}{S_l^2(\eta)}\\
 =&(n-1)\sum_p\frac{S_{k-1}(\eta|p)S_l(\eta)-S_k(\eta)S_{l-1}(\eta|p)}{S_l^2(\eta)}\\
 =&(n-1)\frac{1}{S_l^2(\eta)}\big[(n-k+2)S_{k-1}(\eta)S_l(\eta)-\sigma_{k-1}(\eta)S_l(\eta)\\
 &-(n-l+1)S_k(\eta)S_{l-1}(\eta)-\alpha\sigma_{l-2}(\eta)S_k(\eta)\big]\\
 \leq&(n-1)(n-k+2)\frac{S_{k-1}(\eta)}{S_l(\eta)}.
\end{align*}
If $0\leq l<k<n$, we know $n-l\geq n-k+1\geq2$. Thus by (2) and
(3), we derive $S_{k-1}(\eta|2)\geq
S_{k-1}(\eta|n-k+1)\geq c(n,k)S_{k-1}(\eta)$ and $S_l(\eta|2)\geq
S_l(\eta|n-l)\geq c(n,l)S_l(\eta)$.
\par
Combining the above two inequalities with Lemma~\ref{lem2.5}, we
obtain
\begin{align*}
\frac{\partial
\big[\frac{S_k(\eta)}{S_l(\eta)}\big]}{\partial\lambda_1}=&\sum_{p\neq1}\frac{S_{k-1}(\eta|p)S_l(\eta)-S_k(\eta)S_{l-1}(\eta|p)}{S_l^2(\eta)}\\
=&\sum_{p\neq1}\frac{S_{k-1}(\eta|p)S_l(\eta|p)-S_k(\eta|p)S_{l-1}(\eta|p)}{S_l^2(\eta)}\\
\geq&\frac{(n+1)(k-l)}{k(n-l+1)}\sum_{p\neq1}\frac{S_{k-1}(\eta|p)S_l(\eta|p)}{S_l^2(\eta)}\\
\geq&\frac{(n+1)(k-l)}{k(n-l+1)}\frac{S_{k-1}(\eta|2)S_l(\eta|2)}{S_l^2(\eta)}\\
\geq&c(n,k,l)\frac{S_{k-1}(\eta)}{S_l(\eta)}\geq\sum_j\frac{\partial
\big[\frac{S_k(\eta)}{S_l(\eta)}\big]}{\partial\lambda_j}.
\end{align*}
We therefore obtain (5).
\par
Last, we prove (6). By direct calculating, we have
\begin{align*}
\sum_i\frac{\partial \big[\frac{S_k(\eta)}{S_l(\eta)}\big]}{\partial\eta_i}=&\frac{1}{S_l^2(\eta)}\big[(n-k+1)S_{k-1}(\eta)S_l(\eta)+\alpha\sigma_{k-2}(\eta)S_l(\eta)\\
&-(n-l+1)S_k(\eta)S_{l-1}(\eta)-\alpha S_k(\eta)\sigma_{l-2}(\eta)\big]\\
=&\frac{1}{S_l^2(\eta)}\big[(n-k+1)\sigma_{k-1}\sigma_l-(n-l+1)\sigma_k\sigma_{l-1}\\
&+\alpha(n-k+2)\sigma_{k-2}\sigma_l-\alpha(n-l+1)\sigma_{k-1}\sigma_{l-1}\\
&+\alpha(n-k+1)\sigma_{k-1}\sigma_{l-1}-\alpha(n-l+2)\sigma_k\sigma_{l-2}\\
&+\alpha^2(n-k+2)\sigma_{k-2}\sigma_{l-1}-\alpha^2(n-l+2)\sigma_{k-1}\sigma_{l-2}\big]\\
\geq&\frac{1}{S_l^2(\eta)}\big[\frac{k-l}{k}(n-k+1)\sigma_{k-1}\sigma_l+\alpha\frac{k-l-1}{k-1}(n-k+2)\sigma_{k-2}\sigma_l\\
&+\alpha\frac{k-l+1}{k}(n-k+1)\sigma_{k-1}\sigma_{l-1}+\alpha^2\frac{k-l}{k-1}(n-k+2)\sigma_{k-2}\sigma_{l-1}\big]\\
\geq&c(n,k,l)\frac{S_{k-1}(\eta)}{S_l(\eta)}.
\end{align*}
Combining with Lemma~\ref{lem2.6}, we finally get
\begin{equation*}
\sum_i\frac{\partial\big[\frac{S_k(\eta)}{S_l(\eta)}\big]}{\partial\lambda_i}=(n-1)\sum_i\frac{\partial\big[\frac{S_k(\eta)}{S_l(\eta)}\big]}{\partial\eta_i}\geq
c(n,k,l)\big(\frac{S_k(\eta)}{S_l(\eta)}\big)^{1-\frac{1}{k-l}}\geq
c(n,k,l)f^{1-\frac{1}{k-l}}.
\end{equation*}
We complete the proof.
\end{proof}

\par
Next lemma comes from \cite{B2}.

\begin{lemm}\label{lem2.8}
If~$W=(w_{ij})$~is a real symmetric matrix,
$\lambda_i=\lambda_i(W)$~is one of the eigenvalues $(i=1,\cdots,n)$
and $F=F(W)=f(\lambda(W))$~is a symmetric function
of~$\lambda_1,\cdots,\lambda_n$, then for any real symmetric
matrix~$A=(a_{ij})$~, we have
\begin{eqnarray} \label{e2.12}
\frac{\partial^2F}{\partial w_{ij}\partial w_{st}}a_{ij}a_{st} =
\frac{\partial^2f}{\partial\lambda_p\partial\lambda_q}a_{pp}a_{qq}+2\sum_{p<q}
\frac{\frac{\partial f}{\partial\lambda_p}-\frac{\partial
f}{\partial\lambda_q}}{\lambda_p-\lambda_q}a_{pq}^2.
\end{eqnarray}
Moreover, if $f$ is concave and
$\lambda_1\geq\lambda_2\geq\cdots\geq\lambda_n$, we have
\begin{equation}\label{e2.13}
\frac{\partial f}{\partial\lambda_1}(\lambda)\leq\frac{\partial
f}{\partial\lambda_2}(\lambda)\leq\cdots\leq\frac{\partial
f}{\partial\lambda_n}(\lambda).
\end{equation}
\end{lemm}

\begin{proof}
For the proof of (\ref{e2.12}), see Lemma 3.2 in \cite{GLM2} or
\cite{A}. For the proof of (\ref{e2.13}), see Lemma 2.2 in \cite{A}.
\end{proof}

\section{Proof of Theorem 2}

In this section, we prove the interior a priori Hessian estimate for
(\ref{e1.1}). We use the method similar to \cite{CDH}.
\par
Assume $R=1$, otherwise we consider $\tilde
u(x):=\frac{1}{R^2}u(Rx)$ and the equation $S_k(\tilde\eta)=f(Rx)$
in $B_1(0)$. Let
$\tau(x)=\tau(D^2u(x))=(\tau_1,\cdots,\tau_n)\in\mathbb S^{n-1}$ be
the continuous eigenvector field of $D^2u(x)$ corresponding to the
maximum eigenvalue. We consider the auxiliary function
\begin{equation*}
\phi(x)=\rho(x)g(\frac{|Du|^2}{2})u_{\tau\tau},
\end{equation*}
where $\rho(x)=(1-|x|^2)$ and $g(t)=(1-\frac{t}{A})^{-1/3}$ with
$A=sup|Du|^2>0$. Thus $g'(t)=\frac{1}{3A}(1-\frac{t}{A})^{-4/3}$ and
$g''(t)=\frac{4}{9A^2}(1-\frac{t}{A})^{-7/3}$. Assume $\phi(x)$
attains maximum at $x_0\in B_1(0)$. By rotating the coordinate, we
assume $D^2u(x_0)$ is diagonal with $u_{11}(x_0)\geq
u_{22}(x_0)\geq\cdots\geq u_{nn}(x_0)$. Then
$\tau(x_0)=(1,0,\cdots,0)$. Denote $\lambda_i=u_{ii}(x_0),
\lambda=(\lambda_1,\cdots,\lambda_n)$, and then
$\lambda_1\geq\lambda_2\geq\cdots\geq\lambda_n$. Moreover, the test
function
\begin{equation}\label{e3.1}
\varphi=\log\rho+\log g(\frac{|Du|^2}{2})+\log u_{11}
\end{equation}
attains local maximum at $x_0$. In the following, all the
calculations are at $x_0$. Differentiating (\ref{e3.1}) at $x_0$ once
to get that
\begin{equation}\label{e3.2}
0=\varphi_i=\frac{\rho_i}{\rho}+\frac{g'}{g}\sum_ku_ku_{ki}+\frac{u_{11i}}{u_{11}},
\end{equation}
and hence
\begin{equation}\label{e3.3}
\frac{u_{11i}^2}{u_{11}^2}\leq2\frac{\rho_i^2}{\rho^2}+2\frac{{g'}^2}{g^2}u_i^2u_{ii}^2.
\end{equation}
Differentiating (\ref{e3.1}) twice to see that
\begin{align*}
0\geq\varphi_{ii}=&\frac{\rho_{ii}}{\rho}-\frac{\rho_i^2}{\rho^2}+\frac{g''g-{g'}^2}{g^2}\sum_ku_ku_{ki}\sum_lu_lu_{li}\\
&+\frac{g'}{g}\sum_k(u_{ki}u_{ki}+u_ku_{kii})+\frac{u_{11ii}}{u_{11}}-\frac{u_{11i}^2}{u_{11}^2}\\
\geq&\frac{\rho_{ii}}{\rho}-3\frac{\rho_i^2}{\rho^2}+\frac{g''g-3{g'}^2}{g^2}u_i^2u_{ii}^2+\frac{g'}{g}\Big[u_{ii}^2+\sum_ku_ku_{kii}\Big]+\frac{u_{11ii}}{u_{11}}\\
\geq&\frac{\rho_{ii}}{\rho}-3\frac{\rho_i^2}{\rho^2}+\frac{g'}{g}\Big[u_{ii}^2+\sum_ku_ku_{kii}\Big]+\frac{u_{11ii}}{u_{11}},
\end{align*}
where in the last inequality we used
$g''g-3{g'}^2=\frac{1}{9A^2}\big(\frac{1}{1-\frac{t}{A}}\big)^{8/3}>0$.
\par
Let
\begin{equation*}
F^{ij}=\frac{\partial\frac{S_k(\eta)}{S_l(\eta)}}{\partial
u_{ij}}=\begin{cases}
\dfrac{\partial\frac{S_k(\eta)}{S_l(\eta)}}{\partial\lambda_i},&i=j,\\
0,&i\neq j,
\end{cases}
\end{equation*}
and
\begin{equation*}
F^{ij,rs}=\frac{\partial^2\frac{S_k(\eta)}{S_l(\eta)}}{\partial
u_{ij}\partial u_{rs}}.
\end{equation*}
We have
\begin{equation}\label{e3.4}
\sum_iF^{ii}u_{ii1}=\sum_{i,j}F^{ij}u_{ij1}=f_1+f_uu_1+\sum_i\frac{\partial
f}{\partial u_i}u_{i1},
\end{equation}
and from Lemma \ref{lem2.4},
\begin{align}\label{e3.5}
\sum_iF^{ii}u_{ii11}=&\sum_{i,j}F^{ij}u_{ij11}\geq-\sum_{i,j,r,s}F^{ij,rs}u_{ij1}u_{rs1}-C-Cu_{11}-Cu_{11}^2-|\sum_i\frac{\partial
f}{\partial u_i}u_{11i}|\nonumber\\
\geq&-\Big(1-\frac{1}{k}\Big)\frac{(\sum_iF^{ii}u_{ii1})^2}{f}-C-Cu_{11}-Cu_{11}^2\geq-C-Cu_{11}^2.
\end{align}
In (\ref{e3.5}), we used (\ref{e3.2}) and (\ref{e3.4}). Combining
(\ref{e3.3}) with (\ref{e3.4}) and (\ref{e3.5}), we obtain
\begin{align*}
0\geq&\sum_{i=1}^nF^{ii}\varphi_{ii}\nonumber\\
\geq&\sum_{i=1}^nF^{ii}\Big[\frac{\rho_{ii}}{\rho}-3\frac{\rho_i^2}{\rho^2}\Big]+\frac{g'}{g}\sum_{i=1}^nF^{ii}u_{ii}^2\nonumber\\
&+\frac{g'}{g}\sum_ku_k(f_k+f_uu_k+\sum_i\frac{\partial f}{\partial u_k}u_{kk})-C-Cu_{11}\nonumber\\
\geq&\sum_{i=1}^nF^{ii}\Big[\frac{-2}{\rho}-\frac{12}{\rho^2}\Big]+\frac{1}{3A}F^{11}u_{11}^2-C-Cu_{11}.
\end{align*}
\par
From Lemma \ref{lem2.7} (5), we have
\begin{equation}\label{e3.8}
F^{11}\geq c_0\sum_{i=1}^nF^{ii},
\end{equation}
where $c_0$ is a positive constant depending only on $n$, $k$ and $l$. Then we arrive at
\begin{align*}
0\geq&\sum_{i=1}^nF^{ii}\varphi_{ii}\nonumber\\
\geq&\sum_{i=1}^nF^{ii}\Big[\frac{-2}{\rho}-\frac{12}{\rho^2}\Big]+\frac{1}{3A}c_0\lambda_1^2\sum_{i=1}^kF^{ii}-C-Cu_{11}.
\end{align*}
By Lemma \ref{lem2.4} (6), we get that $\sum F^{ii}\geq c_1$, where
$c_1$ is a positive constant depending only on $n$, $k$, $l$ and the
lower bound of $f$. Therefore, from (\ref{e3.8}), we derive that
\begin{equation*}
\rho(x_0)u_{11}(x_0)\leq C(1+\sup|Du|),
\end{equation*}
where $C$ is a positive constant depending only on $n$, $k$, $l$, $m$ and
$|f|_{C^2}$. Similarly, we have
\begin{align*}
u_{\tau\tau}(0)=&\rho(0)u_{\tau\tau}(0)=\frac{\phi(0)}{g\big(\frac{1}{2}|Du(0)|^2\big)}\nonumber\\
\leq&\frac{\phi(x_0)}{g\big(\frac{1}{2}|Du(0)|^2\big)}=\frac{g\big(\frac{1}{2}|Du(x_0)|^2\big)}{g\big(\frac{1}{2}|Du(0)|^2\big)}\rho(x_0)u_{11}(x_0)\nonumber\\
\leq&C(1+\sup|Du|).
\end{align*}
This shows that (\ref{e1.2}) holds.

\section{Proof of Theorem 4}

\par
In this section, we use the ideas from \cite{CTX} to prove
Theorem 4. For convenience, we suppose $U[u]=(\Delta u)I-D^2u$ and
introduce the following notations.
\begin{equation*}
F(U)=\Big[\frac{S_n(U)}{S_l(U)}\Big]^{\frac{1}{n}},\quad
T(D^2u)=F(U),\quad F^{ij}=\frac{\partial F}{\partial U_{ij}},\quad
F^{ij,rs}=\frac{\partial^2F}{\partial U_{ij}\partial U_{rs}},
\end{equation*}
From the notation we have
\begin{equation*}
T^{ii}=\frac{\partial T}{\partial u_{ii}}=\sum_{p,q}^n\frac{\partial
F}{\partial U_{pq}}\frac{\partial U_{pq}}{\partial
u_{ii}}=\sum_{j=1}^n\frac{\partial F}{\partial U_{jj}}\frac{\partial
U_{jj}}{\partial u_{ii}}=\sum_{j=1}^nF^{jj}-F^{ii}.
\end{equation*}
where we used~$U_{ii}=\eta_i=\sum_{j=1}^nu_{jj}-u_{ii}$ in the last
equation. Suppose ~$u\in C^4(\Omega)\cap C^2(\overline{\Omega})$~is
a solution of equation (\ref{e1.1}), $\eta\in\tilde{\Gamma}_n$.
Without loss of generality, by the maximum principle we assume $u<0$
in $\Omega$. Consider the following test function
\begin{equation*}
\tilde
P(x)=\beta~\log(-u)+\log~\lambda_{max}(x)+\frac{a}{2}|Du|^2+\frac{A}{2}|x|^2,
\end{equation*}
where $\lambda_{max}(x)$ is the biggest eigenvalue of the Hessian
matrix $u_{ij}$, $\beta$, $a$ and $A$ are constants which will be
determined later. Suppose $\tilde P$ attains the maximum value in
$\Omega$ at $x_0$. By rotating the coordinates, we diagonalize the
matrix $D^2u=(u_{ij})$ where
\begin{equation*}
u_{ij}(x_0)=u_{ii}(x_0)\delta_{ij},\quad u_{11}(x_0)\geq
u_{22}(x_0)\geq\cdots\geq u_{nn}(x_0),
\end{equation*}
then
\begin{equation*}
U_{11}(x_0)\leq U_{22}(x_0)\leq\cdots\leq U_{nn}(x_0).
\end{equation*}
Thus by (\ref{e2.13}), we get
\begin{equation}\label{e5.1}
F^{11}(x_0)\geq F^{22}(x_0)\geq\cdots\geq F^{nn}(x_0)>0,
\end{equation}
\begin{equation*}
0<T_{11}(x_0)\leq T_{22}(x_0)\leq\cdots\leq T_{nn}(x_0).
\end{equation*}

\par
Now we define a new function
\begin{equation*}
P(x)=\beta~log(-u)+log~u_{11}(x)+\frac{a}{2}|Du|^2+\frac{A}{2}|x|^2,
\end{equation*}
which also attain the maximum value at $x_0$. Differentiating $P$ at
$x_0$ once to obtain
\begin{equation}\label{e5.3}
\frac{\beta u_i}{u}+\frac{u_{11i}}{u_{11}}+au_iu_{ii}+Ax_i=0.
\end{equation}
Differentiating $P$ at $x_0$ twice to get
\begin{equation*}
\frac{\beta u_{ii}}{u}-\frac{\beta
u_i^2}{u^2}+\frac{u_{11ii}}{u_{11}}-\frac{u_{11i}^2}{u_{11}^2}+a\sum_{p=1}^nu_pu_{pii}+au_{ii}^2+A\leq0.
\end{equation*}
Thus, at $x_0$,
\begin{align}
0 & \geq T^{ii}P_{ii}\notag\\
  & \geq\frac{\beta T^{ii}u_{ii}}{u}-\frac{\beta
  T^{ii}u_i^2}{u^2}+\frac{T^{ii}u_{11ii}}{u_{11}}-\frac{T^{ii}u_{11i}^2}{u_{11}^2}\notag\\\label{e5.5}
  &~~+a\sum_{p=1}^nu_pT^{ii}u_{iip}+aT^{ii}u_{ii}^2+A\sum_{i=1}^nT^{ii}.
\end{align}
\par
We now plan to estimate each term in (\ref{e5.5}). By calculating
directly and using Proposition \ref{lem2.1} (5), we have
\begin{align}
T^{ii}u_{ii}
=&\sum_{i=1}^n\left(\sum_{j=1}^nF^{jj}-F^{ii}\right)\left(\frac{1}{n-1}
               \sum_{p=1}^nU_{pp}-U_{ii}\right)\notag\\
             =&\sum_{i=1}^nF^{ii}U_{ii}+\frac{n}{n-1}\sum_{j=1}^nF^{jj}\sum_{p=1}^nU_{pp}\notag\\
              &-\frac{1}{n-1}\sum_{i=1}^nF^{ii}\sum_{p=1}^nU_{pp}-\sum_{i=1}^{n}\sum_{j=1}^{n}F^{jj}U_{ii}\notag\\
             =&\sum_{i=1}^nF^{ii}U_{ii}\notag\\
             =&\frac{1}{n-l}\Big[\frac{S_n(U)}{S_l(U)}\Big]^{\frac{1}{n-l}-1}\frac{(n-l)S_n(U)S_l(U)-\alpha(\sigma_l(U)\sigma_{n-1}(U)-\sigma_n(U)\sigma_{l-1}(U))}{S_l^2(U)}\notag\\\label{e5.6}
             =&\psi-\alpha Q,
\end{align}
where $\psi=f^{\frac{1}{n-l}}$,
$Q=\frac{\psi(\sigma_l(U)\sigma_{n-1}(U)-\sigma_n(U)\sigma_{l-1}(U))}{(n-l)fS_l^2(U)}$. It follows from (\ref{e2.4}) that $Q>0$.
Rewriting equation (\ref{e1.1}) as
\begin{equation}\label{e5.7}
F(U)=\psi.
\end{equation}
Differentiating equation (\ref{e5.7}) once gives
\begin{equation*}
F^{ii}U_{iip}=\psi_p,
\end{equation*}
which yields
\begin{equation}\label{e5.8}
T^{ii}u_{iip}=\psi_p.
\end{equation}
Differentiating equation (\ref{e5.7}) twice gives
\begin{equation}\label{e5.9}
F^{ij,rs}U_{ijp}U_{rsp}+F^{ii}U_{iipp}=\psi_{pp}.
\end{equation}
We estimates $\psi_{11}$ by (\ref{e5.3}),
\begin{align*}
|\psi_{11}| &=\lvert\psi_{x_1x_1}+\psi_{x_1u}\frac{\partial
             u}{\partial x_1}+\sum_{i=1}^n\psi_{x_1Du_i}\frac{\partial Du_i}{\partial x_1}+\psi_{ux_1}\frac{\partial u}{\partial
             x_1}+\psi_{uu}(\frac{\partial u}{\partial x_1})^2\\
            &~~+\sum_{i=1}^n\psi_{uDu_i}\frac{\partial u}{\partial x_1}\frac{\partial Du_i}{\partial x_1}
             +\psi_{u}\frac{\partial^2u}{\partial x_1^2}+\sum_{i=1}^n\psi_{Du_ix_1}\frac{\partial Du_i}{\partial x_1}
             +\sum_{i=1}^n\psi_{Du_iu}\frac{\partial Du_i}{\partial x_1}\frac{\partial u}{\partial
             x_1}\\
            &~~+\sum_{i=1}^n\sum_{j=1}^n\psi_{Du_iDu_j}\frac{\partial Du_i}{\partial x_1}\frac{\partial Du_j}{\partial x_1}
             +\sum_{i=1}^n\psi_{Du_i}\frac{\partial^2Du_i}{\partial
             x_1^2}\rvert\\
            &\leq C(1+u_{11}+u_{11}^2)+\frac{\partial\psi}{\partial u_i}u_{11i}\\
            &\leq C(1+u_{11}+u_{11}^2)+\frac{C\beta u_{11}}{-u}.
\end{align*}
Then using the concavity of $F$, from (\ref{e5.9}) at $x_0$ we have
\begin{align*}
F^{ii}U_{ii11} &\geq-F^{ij,rs}U_{ij1}U_{rs1}-C(1+u_{11}^2)+\frac{C\beta u_{11}}{u}\\
               &\geq-2\sum_{i=2}^nF^{1i,i1}U_{1i1}^2-C(1+u_{11}^2)+\frac{C\beta u_{11}}{u}\\
               &\geq-2\sum_{i=2}^nF^{1i,i1}u_{11i}^2-C(1+u_{11}^2)+\frac{C\beta
               u_{11}}{u},
\end{align*}
which can be rewritten as
\begin{equation}\label{e5.10}
T^{ii}u_{11ii}\geq-2\sum_{i=2}^nF^{1i,i1}u_{11i}^2-C(1+u_{11}^2)+
\frac{C\beta u_{11}}{u}.
\end{equation}
Plugging (\ref{e5.6}), (\ref{e5.8}) and (\ref{e5.10}) into
(\ref{e5.5}), assuming $u_{11}(x_0)\geq1$, at $x_0$ we have
\begin{align}
0 & \geq T^{ii}P_{ii}\notag\\
  & \geq\frac{C\beta}{u}-\frac{\alpha\beta Q}{u}-\frac{\beta T^{ii}u_i^2}{u^2}-\frac{2}{u_{11}}\sum_{i=2}^nF^{1i,i1}u_{11i}^2-C(1+u_{11}^2)\notag\\
  &~~-\frac{T^{ii}u_{11i}^2}{u_{11}^2}+a\sum_{p=1}^nu_p\psi_p+aT^{ii}u_{ii}^2+A\sum_{i=1}^nT^{ii}\notag\\
  & \geq\frac{C\beta}{u}-\frac{\alpha\beta Q}{u}-\frac{\beta T^{ii}u_i^2}{u^2}-\frac{2}{u_{11}}\sum_{i=2}^nF^{1i,i1}u_{11i}^2-\frac{T^{ii}u_{11i}^2}{u_{11}^2}\notag\\\label{e5.11}
  &~~+aT^{ii}u_{ii}^2+A\sum_{i=1}^nT^{ii}-C(1+u_{11}).
\end{align}

\par
To deal with $-Cu_{11}$ and the third derivative terms, we
divide our proof into two cases and determine the positive constant
$\delta$ later. Note that the letter $C$ denotes the constant that
only depends on the known data, such as $\alpha$, $n$,
$\displaystyle\sup_\Omega |u|$ and $\displaystyle\sup_\Omega|Du|$.
Besides, $C$ may be different from line to line.

\par
\textbf{Case~1.}~~For all~$ i\geq2$ at $x_0$,
$|u_{ii}|(x_0)\leq\delta u_{11}(x_0)$.
\par
We prove the following lemma to deal with~$-Cu_{11}$.
\begin{lemm} \label{lem5.1}
For all $i\geq2$,~$|u_{ii}|(x_0)\leq\delta u_{11}(x_0)$, if we
choose $u_{11}(x_0)$ and $A$ sufficiently large, and $\delta$
sufficiently small, then we have at $x_0$
\begin{equation}\label{e5.12}
\sum_{i=1}^nT^{ii}+Q\geq Cu_{11}.
\end{equation}
\end{lemm}
\begin{proof}
All the following calculation are done at $x_0$. Thus we obtain
\begin{equation*}
|U_{11}|\leq(n-1)\delta u_{11}
\end{equation*}
and
\begin{equation*}
[1-(n-2)\delta]u_{11}\leq U_{22}\leq\cdots\leq
U_{nn}\leq[1+(n-2)\delta]u_{11}.
\end{equation*}
Then choosing $\delta$ sufficiently small and using the inequalities
above, we have
\begin{align*}
\sigma_{k}(\eta) &=\sigma_{k}(\eta|1)+\eta_1\sigma_{k-1}(\eta|1)\\
                 &\geq C_{n-1}^k[1-(n-2)\delta]^{k}u_{11}^{k}-C_{n-1}^{k-1}(n-1)\delta[1+(n-2)\delta]^{k-1}u_{11}^{k}\\
                 &\geq\frac{u_{11}^{k}}{2}
\end{align*}
and
\begin{align*}
\sigma_{k}(\eta) &=\sigma_{k}(\eta|1)+\eta_1\sigma_{k-1}(\eta|1)\\
                 &\leq C_{n-1}^k[1+(n-2)\delta]^{k}u_{11}^{k}+C_{n-1}^{k-1}(n-1)\delta[1+(n-2)\delta]^{k-1}u_{11}^{k}\\
                 &\leq Cu_{11}^{k}.
\end{align*}

Moreover, by Proposition \ref{lem2.1} (4) and (\ref{e2.5}) we obtain
\begin{align*}
\sum_{i=1}^nT^{ii}+Q &=(n-1)\sum_{i=1}^nF^{ii}+Q\\
                     &=\frac{1}{n-1}f^{\frac{1}{n-l}-1}\frac{2S_{n-1}S_l-(n-l+2)S_{l-1}S_n-\sigma_{n-1}\sigma_l+\sigma_n\sigma_{l-1}}{S_l^2}+Q\\
                     &=\frac{1}{n-1}f^{\frac{1}{n-l}-1}\frac{2S_{n-1}S_l-(n-l+2)S_{l-1}S_n}{S_l^2}\\
                     &\geq\frac{2(n-l)}{n(n-1)}f^{\frac{1}{n-l}-1}\frac{~~~~~S_{n-1}}{S_l}.
\end{align*}
Choosing $u_{11}$ sufficiently large, with $n\geq l+2$ we have 
\begin{equation*}
\sum_{i=1}^nT^{ii}+Q\geq Cu_{11}^{n-1-l}\geq Cu_{11},
\end{equation*}
which establishes (\ref{e5.12}).
\end{proof}

\par
Next, we deal with the third derivative terms.
\begin{lemm} \label{lem5.2}
For all $i\geq2$, $|u_{ii}|(x_0)\leq\delta u_{11}(x_0)$ and
$\delta\leq1$, at $x_0$ we have
\begin{align}
\frac{2}{1+\delta}\sum_{i=2}^n\frac{T^{ii}u_{11i}^2}{u_{11}^2}
&\leq-\frac{2}{u_{11}}\sum_{i=2}^nF^{1i,i1}u_{11i}^2+C\frac{\beta^2T^{11}}{u^2}\notag\\\label{e5.13}
&~~+Ca^2\delta^2T^{11}u_{11}^2+CA^2T^{11}.
\end{align}
\end{lemm}
\begin{proof}
All the following calculation are done at $x_0$. Since $\forall
i\geq2$, $|u_{ii}|\leq\delta u_{11}$, we obtain
\begin{equation*}
\frac{1}{u_{11}}\leq\frac{1+\delta}{u_{11}-u_{ii}}.
\end{equation*}
Thus from (\ref{e2.12}),
\begin{align*}
-\frac{2}{u_{11}}\sum_{i=2}^nF^{1i,i1}u_{11i}^2
&=\frac{2}{u_{11}}\sum_{i=2}^nF^{11,ii}u_{11i}^2\\
&\geq-\frac{2}{u_{11}}\sum_{i=2}^n\frac{F^{11}-F^{ii}}{\eta_{11}-\eta_{ii}}u_{11i}^2\\
&=-\frac{2}{u_{11}}\sum_{i=2}^n\frac{T^{ii}-T^{11}}{u_{ii}-u_{11}}u_{11i}^2\\
&\geq\frac{2}{1+\delta}\sum_{i=2}^n\frac{T^{ii}-T^{11}}{u_{11}^2}u_{11i}^2.
\end{align*}
Therefore, we have
\begin{equation}\label{e5.14}
\frac{2}{1+\delta}\sum_{i=2}^n\frac{T^{ii}u_{11i}^2}{u_{11}^2}\leq
-\frac{2}{u_{11}}\sum_{i=2}^nF^{1i,i1}u_{11i}^2+\frac{2}{1+\delta}\sum_{i=2}^n\frac{T^{11}u_{11i}^2}{u_{11}^2}.
\end{equation}
Using Cauchy-Schwarz inequality, from (\ref{e5.3}) at $x_0$ we get
\begin{align}
\sum_{i=2}^n\frac{T^{11}u_{11i}^2}{u_{11}^2}&\leq
3\sum_{i=2}^n\frac{\beta^2T^{11}u_i^2}{u^2}+3a^2\sum_{i=2}^nT^{11}u_i^2u_{ii}^2+3A^2\sum_{i=2}^nT^{11}x_i^2\notag\\\label{e5.15}
&\leq
C\frac{\beta^2T^{11}}{u^2}+Ca^2\delta^2T^{11}u_{11}^2+CA^2T^{11}.
\end{align}
where the first and the second $C$ depend on
$\displaystyle\sup_\Omega|Du|$ and the last $C$ depends on $\Omega$
in the second inequality. Combine (\ref{e5.14}) with (\ref{e5.15}),
we complete the proof.
\end{proof}

\par
Substituting (\ref{e5.12}) and (\ref{e5.13}) into (\ref{e5.11}), at
$x_0$ we have
\begin{align}
0 & \geq T^{ii}P_{ii}\notag\\
  & \geq\frac{C\beta}{u}-\frac{C\beta T^{ii}u_i^2}{u^2}+(\frac{2}{1+\delta}-1)\sum_{i=2}^n\frac{T^{11}u_{11i}^2}{u_{11}^2}\notag\\
  &~~-C\frac{\beta^2T^{11}}{u^2}-Ca^2\delta^2T^{11}u_{11}^2-CA^2T^{11}-\frac{T^{11}u_{111}^2}{u_{11}^2}\notag\\\label{e5.16}
  &~~+aT^{ii}u_{ii}^2+\frac{A}{2}\sum_iT^{ii}+Cu_{11}.
\end{align}
By Cauchy-Schwarz inequality, from (\ref{e5.3}) at $x_0$ we know
\begin{equation}\label{e5.17}
-\frac{T^{11}u_{111}^2}{u_{11}^2}\geq-C\frac{\beta^2T^{11}}{u^2}-Ca^2T^{11}u_{11}^2-CA^2T^{11}
\end{equation}
and
\begin{equation}\label{e5.18}
-\sum_{i=2}^n\frac{\beta T^{ii}u_i^2}{u^2}\geq
-\frac{3}{\beta}\sum_{i=2}^n\frac{T^{ii}u_{11i}^2}{u_{11}^2}
-\frac{Ca^2}{\beta}\sum_{i=2}^nT^{ii}u_{ii}^2-\frac{CA^2}{\beta}\sum_{i=2}^nT^{ii}.
\end{equation}
Substituting (\ref{e5.17}) and (\ref{e5.18}) into (\ref{e5.16})
yields at $x_0$
\begin{align*}
0 &
\geq(\frac{2}{1+\delta}-1-\frac{3}{\beta})\sum_{i=2}^n\frac{T^{11}u_{11i}^2}{u_{11}^2}
+(\frac{a}{2}-\frac{Ca^2}{\beta})\sum_{i=2}^nT^{ii}u_{ii}^2\\
  &~~+(\frac{a}{2}-Ca^2-Ca^2\delta^2)T^{11}u_{11}^2+(\frac{A}{2}-\frac{CA^2}{\beta})\sum_{i=2}^nT^{ii}\\
  &~~+\frac{A}{2}T^{11}-CA^2T^{11}-C\frac{\beta^2T^{11}}{u^2}+\frac{C\beta}{u}+Cu_{11}.
\end{align*}
Fixing the chosen $A$ and $\delta<1$ in Lemma \ref{lem5.1}, choosing
$\beta$ sufficiently large and $a$ sufficiently small, at $x_0$ we
obtain
\begin{equation*}
0\geq\frac{a}{4}T^{11}u_{11}^2-(CA^2+C\frac{\beta^2}{u^2})T^{11}+\frac{C}{u}+Cu_{11},
\end{equation*}
that is,
\begin{equation*}
(CA^2+C\frac{\beta^2}{u^2})T^{11}-\frac{C}{u}\geq\frac{a}{4}T^{11}u_{11}^2+Cu_{11}.
\end{equation*}
We divided the discussion of the inequalities above into two cases. On the left
side, if the first term is less than the second term, we will get
\begin{equation*}
-\frac{2C}{u}\geq Cu_{11},
\end{equation*}
that is,
\begin{equation*}
0\geq\frac{C}{u}+Cu_{11}.
\end{equation*}
If the first term is greater than the second term, we will get
\begin{equation*}
4C\frac{\beta^2}{u^2}\geq\frac{a}{4}u_{11}^2,
\end{equation*}
that is,
\begin{equation*}
C\geq u_{11}^2{(-u)^2}.
\end{equation*}
In conclusion, at $x_0$ we have
\begin{equation*}
(-u)^\beta u_{11}\leq C.
\end{equation*}
Thus we obtain the Pogorelov type $C^2$ estimates for Case 1.

\par
\textbf{Case~2.}~~$u_{22}(x_0)>\delta u_{11}(x_0)$ or
$u_{nn}(x_0)<-\delta u_{11}(x_0)$ at $x_0$.
\par
First, we need the following lemma to deal with $-Cu_{11}$.
\begin{lemm} \label{lem5.3}
If we choose $u_{11}(x_0)\geq\frac{2}{a\delta^2}$, then at $x_0$
\begin{equation}\label{e5.19}
\frac{a}{2}\sum_{i=1}^nT^{ii}u_{ii}^2\geq3Cu_{11}.
\end{equation}
\end{lemm}
\begin{proof}
All the following calculation are done at $x_0$. From
$u_{22}(x_0)>\delta u_{11}(x_0)$ or $u_{nn}(x_0)<-\delta
u_{11}(x_0)$ we know
\begin{equation}\label{e5.20}
\sum_{i=1}^nT^{ii}u_{ii}^2\geq
T^{22}u_{22}^2+T^{nn}u_{nn}^2\geq\delta^2T^{22}u_{11}^2.
\end{equation}
Furthermore, from Lemma \ref{lem2.7} (6) we have
\begin{equation}\label{e5.22}
T^{22}=\sum_{i\neq2}F^{ii}\geq F^{11}\geq\frac{1}{n}\sum_{i=1}^nF^{ii}\geq C.
\end{equation}
Substituting (\ref{e5.22}) into (\ref{e5.20})
gives
\begin{equation*}
\frac{a}{2}\sum_{i=1}^nT^{ii}u_{ii}^2\geq Ca\delta^2u_{11}^2.
\end{equation*}
Choosing $u_{11}\geq\frac{3}{a\delta^2}$, then
\begin{equation*}
\frac{a}{2}\sum_{i=1}^nT^{ii}u_{ii}^2\geq3Cu_{11}, 
\end{equation*}
and therefore (\ref{e5.19}) holds.
\end{proof}

\begin{rema}
Plugging $T^{22}\geq\frac{1}{n}\sum_{i=1}^nF^{ii}$ into
(\ref{e5.20}) yields
\begin{equation}\label{e5.23}
\frac{a}{2}\sum_{i=1}^nT^{ii}u_{ii}^2\geq
ca\delta^2u_{11}^2\sum_{i=1}^nT^{ii},
\end{equation}
This inequality will be used later.
\end{rema}

\par
By (\ref{e5.19}), from (\ref{e5.11}) we obtain
\begin{align}
0 & \geq T^{ii}P_{ii}\notag\\
  & \geq\frac{C\beta}{u}-\frac{\beta T^{ii}u_i^2}{u^2}-\frac{2}{u_{11}}\sum_{i=2}^nF^{1i,i1}u_{11i}^2-\frac{T^{ii}u_{11i}^2}{u_{11}^2}\notag\\\label{e5.24}
  &~~+\frac{a}{2}T^{ii}u_{ii}^2+A\sum_iT^{ii}+Cu_{11}.
\end{align}
Using Cauchy-Schwarz inequality, from (\ref{e5.3}) at $x_0$ we know
\begin{equation}\label{e5.25}
\frac{T^{ii}u_{11i}^2}{u_{11}^2}\leq
C\frac{\beta^2T^{ii}u_i^2}{u^2}+Ca^2T^{ii}u_{ii}^2+CA^2\sum_{i=1}^nT^{ii}.
\end{equation}
Thus, plugging (\ref{e5.25}) and (\ref{e5.23}) into (\ref{e5.24}),
fixing the chosen $\beta$, $A$ and $\delta$ in Case 1 and choosing
$a$ sufficiently small, we get at $x_0$
\begin{align*}
0 & \geq T^{ii}P_{ii}\\
  & \geq\frac{C\beta}{u}-(\beta+C\beta^2)\frac{T^{ii}u_i^2}{u^2}+(\frac{a}{2}-Ca^2)T^{ii}u_{ii}^2\\
  &~~+(A-CA^2)\sum_{i=1}^nT^{ii}+Cu_{11}\\
  & \geq\sum_{i=1}^nT^{ii}(cu_{11}^2-\frac{C}{u^2}-C)+\frac{C}{u}+Cu_{11}
\end{align*}
where we discard the positive term
$-\frac{2}{u_{11}}\sum\nolimits_{i=2}^nF^{1i,i1}u_{11i}^2$ in the
second inequality. At last, similar to the classification discussion
in Case 1, we have
\begin{equation*}
(-u)^\beta u_{11}\leq C.
\end{equation*}
Thus we obtain the Pogorelov type $C^2$ estimates for Case 2.


\bigskip





\bigskip




\begin{thebibliography}{99}
\bibitem{A}
B. Andrews, {\em Contraction of convex hypersurfaces in Euclidean
space.} Calc. Var. Partial Differential Equations 2 (1994), no. 2,
151--171.

\bibitem{B2}
J.M. Ball, {\em Differentiability properties of symmetric and
isotropic functions.} Duke Math. J. 51 (1984), no. 3, 699--728.

\bibitem{CDH}
C.~Q. Chen, W. Dong and F. Han, {\em Interior Hessian estimates for
a class of Hessian type equations.} Calc. Var. Partial Differential
Equations 62 (2023), no.~2, Paper No. 52, 15 pp.

\bibitem{CJ}
J. Chu and H. Jiao, {\em Curvature estimates for a class of Hessian
type equations.} Calc. Var. Partial Differential Equations 60
(2021), no.~3, Paper No. 90, 18 pp.

\bibitem{CNS3}
L.~\'A. Caffarelli, L. Nirenberg and J. Spruck, {\em The Dirichlet
problem for nonlinear second-order elliptic equations. III.
Functions of the eigenvalues of the Hessian.} Acta Math. 155 (1985),
no.~3-4, 261--301.

\bibitem{CTX}
L. Chen, Q. Tu and N. Xiang, {\em Pogorelov type estimates for a
class of Hessian quotient equations.} J. Differential Equations 282
(2021), 272--284.

\bibitem{CW}
K.S. Chou and X.J. Wang, {\em A variational theory of the Hessian
equation}. {Comm. Pure Appl. Math.}, {54}, (2001), 1029--1064.

\bibitem{Dws} W. Dong, {\em Second order estimates for a class of complex Hessian equations on Hermitian
manifolds.} J. Funct. Anal., 281(7), (2021), Paper No. 109121.

\bibitem{D2}
W. Dong, {\em Curvature estimates for $p$-convex hypersurfaces of
prescribed curvature.} Rev. Mat. Iberoam. 39 (2023), no.~3,
1039--1058.

\bibitem{DW}
W. Dong and W. Wei, {\em The Neumann problem for a type of fully
nonlinear complex equations.} J. Differential Equations 306 (2022),
525--546.

\bibitem{FWW}
J. Fu, Z. Wang and D. Wu, {\em Form-type Calabi-Yau equations.}
Math. Res. Lett. 17 (2010), no.~5, 887--903.

\bibitem{G3}
P. Gauduchon, {\em La $1$-forme de torsion d'une vari\'et\'e{}
hermitienne compacte.} Math. Ann. 267 (1984), no.~4, 495--518.

\bibitem{GLL}
P. Guan, J. Li and Y.~Y. Li, {\em Hypersurfaces of prescribed
curvature measure.} Duke Math. J. 161 (2012), no.~10, 1927--1942.

\bibitem{GLM2}
S.Z. Gao, H.Z. Li, H. Ma, {\em Uniqueness of closed self-similar
solutions to $\sigma_k^\alpha$-curvature flow.} NoDEA Nonlinear
Differential Equations Appl. 25 (2018), no. 5, Paper No. 45, 26 pp.

\bibitem{GZ}
P. Guan and X. Zhang, {\em A class of curvature type equations.}
Pure Appl. Math. Q. 17 (2021), no.~3, 865--907.

\bibitem{HL1}
F.~R. Harvey and H.~B. Lawson Jr., {\em Dirichlet duality and the
nonlinear Dirichlet problem on Riemannian manifolds.} J.
Differential Geom. 88 (2011), no.~3, 395--482.

\bibitem{HL2}
F.R. Harvey and H.B. Lawson, {\em Geometric plurisubharmonicity and
convexity: an introduction.} Adv. Math. 230 (2012), no. 4-6,
2428--2456.

\bibitem{HL5}
F.~R. Harvey and H.~B. Lawson Jr., {\em $p$-convexity,
$p$-plurisubharmonicity and the Levi problem.} Indiana Univ. Math.
J. 62 (2013), no.~1, 149--169.

\bibitem{JW}
H. Jiao and Z. Wang, {\em Second order estimates for convex
solutions of degenerate $k$-Hessian equations.} J. Funct. Anal. 286
(2024), no.~3, Paper No. 110248, 30 pp.

\bibitem{LR}
Y. Liu, C. Ren, {\em Pogorelov type C2 estimates for Sum Hessian
equations and a rigidity theorem.} Journal of Functional Analysis,
Volume 284, Issue 1, 2023, 109726, ISSN 0022-1236.

\bibitem{LRW}
C. Li, C. Ren and Z. Wang, {\em Curvature estimates for convex
solutions of some fully nonlinear Hessian-type equations.} Calc.
Var. Partial Differential Equations 58 (2019), no. 6, Paper No. 188,
32 pp.

\bibitem{M}
C. Maclaurin, {\em A second letter to Martin Folkes, Esq.:
concerning the roots of equations, with the demonstration of other
rules in algebra.} Phil. Transactions, 36 (1729), 59--96.

\bibitem{N}
I. Newton, {\em Arithmetica universalis, sive de compositione et
resolutione arithmetica liber.} 1707.

\bibitem{Q}
G. Qiu, {\em Interior Hessian estimates for $\sigma_2$ equations in
dimension three.} Front. Math. 19 (2024), no.~4, 577--598.

\bibitem{R}
C.~Y. Ren, {\em A generalization of Newton-Maclaurin's
inequalities.} Int. Math. Res. Not. IMRN (2024), no.~5, 3799--3822.

\bibitem{RW}
C.~Y. Ren and Z. Wang, {\em Interior $C^2$ estimates for a class of 
sum Hessian equations.} J. Differential Equations 446 (2025), 
Paper No. 113631, 21 pp.

\bibitem{S2}
J.-P. Sha, {\em $p$-convex Riemannian manifolds.} Invent. Math. 83
(1986), no.~3, 437--447.

\bibitem{S3}
J.-P. Sha, {\em Handlebodies and $p$-convexity.} J. Differential
Geom. 25 (1987), no.~3, 353--361.

\bibitem{S4}
J. Spruck, {\em Geometric aspects of the theory of fully nonlinear 
elliptic equations}. Clay Math. Proc. 2 (2005) 283-309.

\bibitem{STW}
G. Sz\'ekelyhidi, V. Tosatti and B. Weinkove, {\em Gauduchon metrics
with prescribed volume form.} Acta Math. 219 (2017), no.~1,
181--211.

\bibitem{TW2}
V. Tosatti, B Weinkove, {\em The Monge-Amp\`ere equation for
(n-1)-plurisubharmonic functions on a compact K$\ddot{a}$hler
manifold.} J. Amer. Math. Soc. 30 (2017), no. 2, 311--346.

\bibitem{W}
H.~H. Wu, {\em Manifolds of partially positive curvature.} Indiana
Univ. Math. J. 36 (1987), no.~3, 525--548.

\bibitem{WY}
M.~W. Warren and Y. Yuan, {\em Hessian estimates for the sigma-2
equation in dimension 3.} Comm. Pure Appl. Math. 62 (2009), no.~3,
305--321.

\bibitem{Zhou} J. Zhou, {\em The interior gradient estimate for a class of mixed
Hessian curvature equations,} J. Korean Math. Soc., 59(1), (2022),
53-69.

\end{thebibliography}
\end{document}